\documentclass{article}

\usepackage{amsmath, amsfonts, amssymb, setspace, xypic}
\input xy
\xyoption{all}
\newtheorem{theorem}{Theorem}[section]
\newtheorem{lemma}[theorem]{Lemma}
\newtheorem{lemma-def}[theorem]{Lemma-Definition}
\newtheorem{proposition}[theorem]{Proposition}
\newtheorem{corol}[theorem]{Corollary}
\newtheorem{definition}[theorem]{Definition}

\newcommand{\Q}{\mathbb{Q}}

\newcommand{\F}{\mathbb{F}}
\newcommand{\FF}{\mathbb{F}}
\newcommand{\Z}{\mathbb{Z}}
\newcommand{\ZZ}{\Z}
\newcommand{\N}{\mathbb{N}}

\newcommand{\K}{\mathbb{K}}
\newcommand{\PP}{\mathbb{P}}
\newcommand{\Lm}{\mathcal{L}}
\newcommand{\cM}{\mathcal{M}}
\newcommand{\cO}{\mathcal{O}}

\newcommand{\ord}{\mathrm{ord}\,}

\newcommand{\Lring}{\Lm_{\text{ring}}}

\newcommand{\Ldist}{\Lm_{\text{dist}}}
\newcommand{\Deltadist}{\Delta_{\text{dist}}}
\newcommand{\LmM}{\Lm_{M}}

\newcommand{\PlusR}{\boxplus_R}
\newcommand{\MinR}{\boxminus_R}

\newcommand{\Dvier}{D^{(4)}}
\newcommand{\Dvierk}{\Dvier_k}
\newcommand{\ac}{\text{ac}\,}

\newcommand{\mathor}{\quad \text{or}\quad}

\newenvironment{proof}{\noindent\textit{Proof.}} {\hfill{$\square$}\newline}
\DeclareMathOperator{\sq}{\square}

\title{Cell decomposition and definable functions for weak $p$-adic structures}
\author{Eva Leenknegt\\
Purdue University, Department of Mathematics, \\150 N. University Street, West Lafayette, IN 47907-2067, USA\\ 
\texttt{eleenkne@math.purdue.edu}}
\date{}
\begin{document}
\maketitle
\begin{abstract}
We develop a notion of cell decomposition suitable for studying weak $p$-adic structures (reducts of $p$-adic fields where addition and multiplication are not (everywhere) definable).
\end{abstract}
\section{Introduction and first definitions}

Results for real fields have always been a big source of inspiration for the study of $p$-adic fields.  An example of this is the concept of $o$-minimality, see e.g van den Dries \cite{vdd-98}, which inspired Haskell and Macpherson \cite{has-mac-97} to develop a similar concept, $P$-minimality, for $p$-adic fields. A difference between those concepts is that $o$-minimality also covers reducts of real closed fields, see Peterzil \cite{mpp-92, pet-92, pet-93}, while $P$-minimality focuses on expansions of the language of valued fields.
To fill this gap in the study of $p$-adic fields, we need to consider reducts $(K, \Lm)$  of $(K, \Lring)$, where $K$ is a $p$-adically closed field, and  the $\Lm$-definable subsets of $K$ are exactly the $\Lring$-definable (semi-algebraic) subsets of $K$. 

A first step towards understanding such structures is to describe the boundaries of our `playing field': identify the relations and functions that, as a bare minimum, would have to be definable in such a structure.
In our paper \cite{clu-lee-2011}, we concluded that any reduct of $(K, \Lring)$ where the relations \[R_{n,m}(x,y,z) \leftrightarrow y-x \in z Q_{n,m}\] are definable, would fit inside this framework. The sets $Q_{n,m}$, which for $K = \Q_p$ can be defined as $\cup_{k\in \Z} p^{kn}(1 + p^m \Z_p)$, are a variation on the sets of $n$-th powers $P_n$ that one encounters in the study of $p$-adic semi-algebraic sets. A more general definition will be given in section \ref{subsec:defqnm}.

Now that we know which structures we want to consider, the second step will be to describe their definable sets. 
Historically, cell decomposition has proved to be a very useful tool in studying definability questions. Examples include $o$-minimal cell decomposition in the real case, and Denef's cell decomposition for $p$-adic semi-algebraic sets, which can be stated as follows:
\begin{theorem}[Denef, \cite{denef-84, denef-86}]
Let $K$ be a finite field extension of $\Q_p$. Any semi-algebraic set $X\subseteq K^{k+1}$  can be partitioned as
a finite union of cells of the form
\[\{(x,t) \in D \times K \ | \ \ord a_1(x) \ \square_1 \ \ord (t-c(x)) \ \square_2 \ \ord a_2(x), t-c(x) \in \lambda P_{n}\},\]
where $D$ is a semi-algebraic subset of $K^k$ and $c(x), a_i(x)$ are semi-algebraic functions.
\end{theorem}
Among other applications, Denef used this result to give a new proof of Macintyre's quantifier elimination result \cite{mac-76}. The result was also important for Mourgues \cite{mou-09} result on $P$-minimal cell decomposition.

 In \cite{clu-lee-2011}, we showed that the relation $D^{(3)}(x,y,z) \leftrightarrow \ord (x-y) < \ord (z-y)$ is definable  in the language $\LmM:= (\{R_{n,m}\}_{n,m})$. 
 Because of this, cells  $C \subset K$ are $\LmM$-definable, which is what we wanted.
However, to get a language that is more convenient to work with, we will replace $D^{(3)}$  by the slightly stronger relation $D^{(4)}$, defined as \[D^{(4)}(x,y,z,t) \leftrightarrow \ord (x-y) < \ord (z-t).\] The resulting language $\Ldist = (D^{(4)}, \{R_{n,m}\}_{n,m})$ is strictly stronger than $\LmM$, as there are sets which are $\Ldist$-definable, but not $\LmM$-definable \cite{PhD}. 
\\\\
Our definition of cells is inspired by Denef's $p$-adic cells, but with some modifications, the first being that we use the sets $Q_{n,m}$ instead of the usual $P_n$.
A second difference is that we will only require  the relation $\ord a_i(x) <\ord t$ to be definable, and not necessarily the function $a_i(x)$ itself. This change is motivated by the observation that the function $(x,y) \mapsto x-y$ is not necessarily definable in all languages that contain a symbol for the relation $D^{(4)}$. We will call this \emph{weak} cell decomposition as opposed to (strong)  cell decomposition results that only use definable functions. 

While working on the cell decomposition results presented in this paper, we noticed that many of those results were valid for a much wider class of fields than just $p$-adically closed fields. In particular, we do not need to assume that the field is henselian, and our results will work in any characteristic. For this reason, even though our original motivation was the study of $p$-adically closed fields, we present our results for $(\F_q,\Z)$-fields: valued fields with residue field isomorphic to $\F_q$ and value group elementarily equivalent to $\Z$. So this paper is really about cell decomposition techniques for expansions of $(K, \Ldist)$, where $K$ is an $(\F_q, \Z)$-field.
\\\\
Let us now give a brief overview of the contents of this paper. We first explain the ideas behind  the sets $Q_{n,m}$ in section \ref{subsec:defqnm}, and then give a formal definition of our concept of cells in section \ref{sec:cell-def}.
In section \ref{section:celdec}, we state some cell decomposition results valid for all expansions of $\Ldist$, and we show how these results can be used to study the language $\Ldist$ itself.

In our definition of cells, we made a distinction between weak and strong cell decomposition, depending on whether or not the functions used were definable. 
 In section \ref{section:deffun} we investigate the definable functions of structures that admit weak cell decomposition. In particular, we will focus on the existence of definable Skolem functions. The reason for this is a result by Mourgues for $P$-minimal fields, stating that a structure has (strong)  cell decomposition if and only if it admits definable Skolem functions. When restricting our attention to $p$-adically closed fields, and under the  assumption that  multiplication by a sufficient number of scalars is definable, we can obtain a similar result for expansions of $\Ldist$.
 
 However, this result is not as strong as it seems to be. In $P$-minimality, requiring the existence of definable Skolem functions is a relatively minor assumption, as there are no known examples of structures that do not have such functions. For the weaker structures we study, we get a different picture: $(K,\Ldist)$ itself is an example of a structure having no definable Skolem functions. So the really interesting questions are whether or not every expansion of  $(K,\Ldist)$ admits weak cell decomposition, and under which conditions a structures would have definable skolem functions.
 
 At this time, we cannot answer the first question, and we can only give a conjecture for the second question. In $o$-minimality, it is known that any structure where addition is definable, has definable skolem functions. Based on the structures we have studied, our guess would be that an expansion of $\Ldist$ probably has Skolem functions if addition and sufficient scalar multiplication is definable. 
 While we cannot give a proof for this, we have strong indications that structures where these requirements are not satisfied can never have definable Skolem functions.  Section \ref{section:subaffine}  provides an example of a structure where we have full scalar multiplication, and where addition is definable on large open sets, but which still does not admit definable Skolem functions.

 The expansions considered in this paper are still very basic. Adding a symbol for either addition \cite{lee-2011} or (restricted) multiplication \cite{PhD} is also  possible, but for this we refer to \cite{lee-2011} and \cite{PhD}.

\subsection{Definition of the sets $Q_{n,m}$} \label{subsec:defqnm}
The field of $p$-adic numbers admits elimination of quantifiers in the language $\Lm_{\text{Mac}}$, which is the ring language, extended with predicates $P_n$ for the sets of $n$-th powers. 
Note that by Hensel's lemma, any $x \in \Q_p$ can be written, using a sufficiently large $r \in \N$,  as $x = \lambda \cdot u$, with $ u \in (1 + p^r\Z_p) \cap P_n$, so that cosets $\lambda P_n$ encode
certain information concerning the value group and the angular components of their elements. This is essentially the reason why this language admits elimination of quantifiers.

However, this only works because $\Q_p$ and other $p$-adically closed fields are Henselian. In general, additive-multiplicative congruences (amc-structures), as proposed by Basarab and Kuhlmann  \cite{bas-kuh-92, kuh-94} are necessary to obtain quantifier elimination in a definitional expansion of the valued field language. Amc-structures are quotient groups  $K^{\times}/(1+I)$, for ideals $I \subset \mathcal{O}_K$ 
(where $\mathcal{O}_K$ is the valuation ring of $K$). 
For example, if $I = \mathcal{M}_K$, the maximal ideal of $\mathcal{O}_K$, then the corresponding amc-structure encodes information about the value group and the angular component modulo $\pi$. 

We propose to use a variation on amc-structures, that encodes similar information as the sets $P_{n}$, even when Hensel's lemma does not hold. In particular, we will consider the following class of fields.

\begin{definition}
Let $\FF_q$ be the finite field with $q$ elements and $\ZZ$ the ordered abelian group of integers. We define a  $(\FF_q,\ZZ)$-field  to be a valued field with residue field
isomorphic to $\FF_q$ and value group elementarily equivalent to $\ZZ$.
\end{definition}
Fix a $(\FF_q,\ZZ)$-field  $K$, fix an element $\pi_K$ with smallest positive order, and write $\cM_K$ for the maximal ideal of the valuation ring $\cO_K$ of $K$. For each integer $n>0$, let $P_n$ be the set of nonzero $n$-th powers in $K$. For each $m>0$, we will define sets $Q_{n,m}$ using angular component maps. The following lemma shows that such maps exist and that they can be defined in a unique way.
\begin{lemma} For each integer $m>0$, there is a unique group homomorphism
$$
\ac_m:K^\times \to (\cO_K\bmod \pi_K^m)^\times
$$
such that $\ac_m(\pi_K)=1$ and such that $\ac_m(u)\equiv u \bmod (\pi_K)^m$ for any unit $u\in\cO_K$.
\end{lemma}

\begin{proof}
Put $N_m:=(q-1)q^{m-1}$ and let $U$ be the set $P_{N_m}\cdot \cO_K^\times$.
Note that $K^\times$ equals the finite disjoint union of the sets $\pi_K^\ell\cdot U$ for integers $\ell$ with $0 \leqslant \ell \leqslant N_m-1$.
Hence, any element $y$ of $K^\times$ can be written as a product of the form $\pi_K^\ell x^{N_m} u$, with $u\in \cO_K^\times$, $\ell\in\{0,\ldots,N_{m}-1\}$, and $x\in K^\times$.

Since $\ac_m$ is required to be a group homomorphism to a finite group with $N_m$ elements, it must send $P_{N_m}$ to $1$. Also note that the projection $\cO_K \to \cO_K\bmod \pi_K^m$ (which is a ring homomorphism), induces a natural group homomorphism $p: \cO_K^\times \to (\cO_K\bmod \pi_K^m)^\times$. Now if we write $y=\pi_K^\ell x^{N_m} u$, we see that $\ac_m$ must satisfy
\begin{equation}\label{acm}
\ac_m ( y ) = p(u),
\end{equation}
which implies that the map $\ac_m$ is uniquely determined if it exists. Moreover, we claim that we can use \eqref{acm} to define $\ac_m$. This is certainly a well defined group homomorphism:
if one writes $y=\pi_K^\ell \tilde{ x}^{N_m} \tilde{u}$ for some other $\tilde{u}\in \cO_K^\times$ and $\tilde{x}\in K^\times$, then clearly $p(u)=p(\tilde {u})$. It is also clear that this homomorphism  sends $\pi_K$ to $1$ and  satisfies our requirement that $\ac_m(u)\equiv u \bmod (\pi_K)^m$ for any unit $u\in\cO_K$. 
\end{proof}

\noindent Using these angular component maps, we can define sets $Q_{n,m}$, for any $m,n >0$, as follows:
$$
Q_{n,m} := \{x\in  P_n \cdot (1 + \cM_K^m)\mid \ac_m(x)=1\}.
$$
Note that $Q_{n,m}$ is an open subgroup of finite index of $K^\times$ (for the valuation topology), and that $Q_{n,m}$ is definable in the language of valued fields $(+,-,\cdot,|)$ by the above construction of $\ac_m$. 
For example if $K = \Q_p$, then $Q_{n,m}$ is just the set $\bigcup_{k\in \Z} \ p^{kn}(1+p^m\Z_p)$.

For any element $\lambda \in K$, let $\lambda Q_{n,m}$ denote the set $\{\lambda t\mid t\in Q_{n,m}\}$.
We will sometimes use the alternative notation $\rho_{n,m}(x) = \lambda$ to express that $x \in \lambda Q_{n,m}$.
The relation between $\rho_{n,m}(x+y)$, $\rho_{n,m}(x)$ and $\rho_{n,m}(y)$ is investigated in the lemma below (the proof is left to the reader).

\begin{lemma}\label{lemma:rhoab}
Put $\delta \in \{-1,1\}$. Suppose that $\rho_{n,m}(a) = \lambda$ and $\rho_{n,m}(b) = \mu$, then $\rho_{n,m}(a+\delta b)$ equals
\[ \left\{\begin{array}{lcl}
\hspace{-4pt}\lambda &\text{if}& m + \ord a \leqslant \ord b, \\
\hspace{-4pt} \rho_{n,m}(\lambda +\delta \mu \pi^{rn}), \text{with}  \ rn = \ord (\frac{\mu a}{\lambda b})& \text{if} & -m +\ord b <\ord a <\ord b,\\
\hspace{-4pt} \rho_{n,m}(\lambda + \delta \mu)&\text {if}& \ord a = \ord b = \ord (a + \delta b).
\end{array}\right.\]
\end{lemma}

\subsection{Definition of weak $p$-adic cells} \label{sec:cell-def}
Let $K$ be an $(\F_q, \Z)$-field. As stated in the introduction, we will be working with functions
 $f: K^k\to K$ for which the relation
\[\ord f(x) < \ord t\] 
is definable for all $(x,t) \in K^{k+1}$. We call these \emph{order-definable} functions.  
Note that $f$ is not required to be a definable function. However,  the following relations are always definable if the relation $\ord x < \ord y$ is definable in our language:

\begin{lemma} \label{lemma:cd-defrel} Let $\Lm$ be a language where the relation $D^{(2)}(x,y) := \ord x < \ord y$ is definable. Let $a_1(x), a_2(x)$ be two functions $K^k \to K$ that are order-definable in $\Lm$. The following 
relations are definable:
\begin{enumerate}
\item[(1)] $\ord x < \ord y +k$, for any $k \in \Z$;
\item[(2)] $\ord a_1(x) \ \square \ \ord t$;
\item[(3)] $\ord a_1(x) \ \square \ \ord a_2(x) +k$, for any $k \in \Z$;
\end{enumerate}
where $\square$ may denote $<,\leqslant, =,\geqslant or >$.
\end{lemma}
\begin{proof}
(1)  For $k <0$, the relation $\ord x +k > \ord y$ is equivalent with
\[
(\exists u_1) \ldots
(\exists u_{-k})[\ord x > \ord u_1 > \ord u_2 > \ldots > \ord u_{-k} > \ord
y].\]
 For $k>0$, it is equivalent with
\[
 (\forall u_1) \ldots (\forall
u_{k})\left[(\ord x < \ord u_1 < \ldots
<\ord u_k)
\rightarrow\ord u_k > \ord y
\right].
\]
(2) For example if $\square$ denotes `$=$', we can define the relation $\ord a_1(x) = \ord t$ as
\[ \neg (\ord a_1(x) < \ord t) \wedge (\forall u)(\ord t < \ord u \rightarrow \ord a_1(x) <\ord u).\]
The other cases can be derived from this.\\
(3) The relation $\ord a_1(x) < \ord a_2(x)$ is equivalent with
\[(\exists t)(\ord a_1(x) <\ord t \wedge \ord a_2(x) = \ord t).\] The rest can be derived from (1) and (2).
\end{proof}
For our notion of cell decomposition, a $p$-adic cell will be a  set of the following form. 
\begin{definition}[$p$-adic cell]\label{def:cell} Let $\Lm$ be an expansion of ($D^{(2)}, \{R_{n,m}\}_{n,m})$. For each $k>0$, let $\Delta_k$ be a collection of (not necessarily $\Lm$-definable) functions $K^k \to K$, and put $\Delta = \cup_k \Delta_k$. \\
An \emph{$(\Lm, \Delta$)-precell} in $K^k$ is a set $\{x \in K^k \ | \ \phi(x) \}$, where $\phi(x)$ is a boolean combination of relations of the forms 
\[\ord a_1(x) \ < \ \ord a_2(x) + \ell , \quad \text{and} \quad  b_1(x) -b_2(x) \in \lambda Q_{n,m},\]
where the $a_i(x)$ are functions in $\Delta_k$, $\ell \in \Z$ and the $b_i(x)$ are quantifier free definable functions for $\Lm$.\\
A \emph{($p$-adic) ($\Lm, \Delta$)-definable cell} $C_c^D(a_1,a_2,\lambda) \subseteq K^{k+1}$ is a set of the following form:
\[\{(x,t) \in D \times K \ | \ \ord a_1(x) \ \square_1 \ \ord t-c(x) \ \square_2 \ \ord a_2(x), t-c(x) \in \lambda Q_{n,m}\},\]
where $\lambda \in K$, $D$ is an ($\Lm, \Delta$)-precell in $K^k$,  $\square_i$ denotes `$<$' or `no condition', and the $a_i(x)$ are functions from $\Delta_k$. We call the function $c(x)$ the center of the cell and we require $c(x)$ to be a quantifier free definable function for $\Lm$.
\end{definition}
\begin{definition}
Let $\Lm$ be a language expanding $(D^{(2)}, \{R_{n,m}\}_{n,m})$, and $K$ an $(\F_q,\Z)$-field, which we consider as an $\Lm$-structure. \\We say that
 $\Lm$ \emph{allows  cell decomposition (with definable centers)} in $K$ if there exists a set of functions $\Delta$ such that for every $k\in \N$, any $\Lm$-definable subset of $K^{k+1}$ can be partitioned into ($\Lm$-definable) $p$-adic ($\Lm,\Delta$)-cells.
 More specifically,
 \begin{itemize}
 \item The decomposition is \emph{strong} if every $f \in \Delta$ is $\Lm$-definable.
 \item The decomposition is \emph{weak} if every $f \in \Delta$ is order-definable for $\Lm$.
\end{itemize}
\end{definition}
The languages we study in this paper only allow weak cell decomposition. Cell decomposition for semi-algebraic sets and semilinear sets are (after a few straightforward adaptations) examples of strong cell decomposition.

\section{Cell Decomposition results} \label{section:celdec}

Let $\Lm$ be a language expanding the language $\Lm_{\text{dist}}$ (which we defined in the introduction).
Every boolean combination of weakly $\Lm$-definable cells can be partitioned into a finite number of weakly $\Lm$-definable cells. This is an easy consequence of the following theorem.

\begin{theorem}\label{theorem:lmcell-doorsnede}
Let $\Lm$ be a language expanding $\Lm_{\text{dist}}$. Let $\Delta$ be a collection of functions that are order-definable in $\Lm$. Let $A_1$, $A_2$ be  $(\Lm, \Delta)$-definable cells with centers $c_1$, resp. $c_2$.
The intersection $A_1 \cap A_2$ can be written as a finite union of disjoint weak
$\Lm$-cells $A$ each of which has a  center which is a restriction of either $c_1$ or $c_2$.
\end{theorem}
\begin{proof}
By partitioning $C_1$ and $C_2$ further if necessary, we may suppose that they both use $Q_{n,m}$ with the same positive integers $m,n$, that is, that $C_i$ is of the form
$$
 \{(x,t)\in D_i \times K \ | \ord a_{1i}(x) \ \sq_{1i} \
\ord(t-c_i) \ \sq_{2i} \ \ord a_{2i}(x), \ t-c_i \in \lambda_i
Q_{n,m}\}
$$
for $i\in\{1,2\}$, where the symbols have their meaning as in Definition \ref{def:cell}.
Using Lemma \ref{lemma:cd-defrel}, we can find a finite partitioning of $C_1$ in cells with the same center,  such that  on such a cell one of the following conditions holds for $k= 1+ m+n + \sum_{i,j=1,2}|k_{ij}|$ and some integer $\ell_1$ with $-k\leq \ell_1\leq k$. 
$$
\begin{array}{cccc}
\ord(t-c_1) &>& \ord(c_2-c_1)+k ,  & ({\rm I})_k\\
\ord(t-c_1)  &<& \ord(c_2-c_1) - k ,  & ({\rm II})_k\\
\ord(t-c_1) + \ell_1 & = & \ord(c_2-c_1).  & ({\rm III})_{\ell_1}
\end{array}
$$
Hence, we may suppose that one of these conditions holds for $C_1$. Note that $({\rm I})_k$ and $({\rm II})_k$  imply respectively
$$
\begin{array}{cccccc}
\ord(t-c_1) &>& \ord(c_2-c_1)+k & =  & \ord ( t - c_2) +k  ,  & ({\rm i})_k\\
\ord(t-c_2) &=& \ord(t-c_1) &<& \ord(c_1-c_2) - k.   & ({\rm ii})_k\\
\end{array}
$$
If $({\rm I})_k$ holds on $C_1$, put
$$
W = \{x \in D_2 \
| \ \ord a_{12}(x) \sq_{12} \ord (c_1 - c_2) \sq_{22} \ord a_{22}(x)\}.
$$
Then one has, if $C_{1} \cap C_{2}$ is nonempty, that
\[C_{1} \cap C_{2} = ( W\times K) \cap C_1,\]
which can easily be seen to be a finite disjoint union of $(\Lm, \Delta)$-cells of the desired form.
If $({\rm II})_k$ holds on $C_{1}$, then we may suppose, up to partitioning $C_2$, that $({\rm II})_k$ holds for all $(x,t) \in C_1$ and all $(x,t) \in C_2$. Also, we find that $\rho_{n,m}(t-c_1) = \rho_{n,m}(t-c_2)$, so either $C_1 \cap C_2$ is empty or $C_1 \cap C_2$ consists of
all points $(x,t) \in (D_1\cap D_2) \times K$ satisfying the
conditions
$$
\max_{i\in I}\{\ord a_{1i}(x)\} < \ord (t-c_1)< \min_{i=1,2} \{\ord a_{2i}(x)\},
$$
and
$$
t-c_1 \in \lambda_1 Q_{n,m},
$$
where $I$ consists of $i$ such that  $\sq_{1i}$ is the condition $<$, and where the maximum over the empty set is $-\infty$. 
We know by Lemma \ref{lemma:cd-defrel} that relations of the form $\ord a_{ij} < \ord a_{lk}$ are definable. It is then easy to see that $C_1 \cap C_2$ can be partitioned into a finite number of disjoint $(\Lm,\Delta)$-cells, which finishes the proof for this case.
\\\\
We may suppose by symmetry (that is, up to reversing the role of $C_1$ and $C_2$) that, if $({\rm III})_{\ell_1}$ holds on $C_1$, then also
$$
\begin{array}{cccccr}
\ord(t-c_1) + \ell_1 & = & \ord(c_2-c_1) &=& \ord ( t-c_2) + \ell_2  & ({\rm iii})_\ell
\end{array}
$$
holds with $\ell=(\ell_1,\ell_2)$ and $-k\leq \ell_2\leq k$. Suppose again that $C_1\cap C_2$ is nonempty. If one now fixes the residue classes of $c_2-c_1$ and of $t-c_1$ modulo $Q_{2kn,2kn}$, then the conditions
$$
\ord(c_2-c_1)=\ord ( t-c_2) + \ell_2\mbox{ and }
t-c_2\in\lambda_2 Q_{n,m}
$$
follow automatically from $\ord(t-c_1) + \ell_1 = \ord(c_2-c_1)$. (The exact relations are described in Lemma \ref{lemma:rhoab}.) Hence, one can easily partition $C_1\cap C_2$ into finitely many $(\Lm,\Delta)$-cells.
\end{proof}

\noindent One of our main motivations for using cell decomposition is because it is a very useful tool for quantifier elimination. An example is Denef's proof of quantifier elimination for semi-algebraic sets \cite{denef-86}. The following lemma, which is closely inspired by this paper, will be used quite often.

\begin{lemma}\label{cd-lemma:qe} Let $\Lm$ be a language expanding $(D^{(2)}, \{R_{n,m}\}_{n,m})$. \\Let $C_c^D(a_1,a_2,\lambda) \subseteq K^{k+1}$ be a weakly $\Lm$-definable cell. Suppose that for every $l\in \N$, the set $\{x\in K^k \ | \ \ord a_1 \equiv l \mod n\}$ can be partitioned as a finite union of precells $\subseteq K^k$.
Then the projection
\[P:=\{x\in K^k  \ | \ \exists t: (x,t) \in C_c^D(a_1,a_2, \lambda)\}\]
can be partitioned in a finite number of precells.
\end{lemma}
\begin{proof}
Note that $P$ is in fact equal to
the following set:
\[ P = \{ x \in D \ | \ \exists \gamma \in \Gamma_K: \ord a_1(x) < \gamma< \ord
a_2(x), \ \gamma \equiv \ord \lambda \mod n\}.\] Thus
$P$ is the set of all $x \in D$ satisfying
\begin{equation}\exists \gamma \in \Gamma_K: \frac{\ord
[a_1(x)\lambda^{-1}]}{n} < \gamma <
\frac{\ord [a_2(x)\lambda^{-1}]}{n}.\label{condforz}\end{equation}
Now if $\ord a_1(x)\lambda^{-1} \equiv \zeta
\mod n$, for $0 \leqslant \zeta <n$, then condition
(\ref{condforz}) is equivalent with
$ \ord a_1(x)\lambda^{-1} + n - \zeta <
\ord a_2(x)\lambda^{-1},$ which can be
simplified to
\begin{equation} \label{condforz2}\ord a_1(x) + n - \zeta <
\ord a_2(x).\end{equation}
This completes the proof, since $D$ is a precell, (\ref{condforz2}) is a precell condition and by our assumption, the set  $\{x\in K^k \ | \ \ord a_1(x)\lambda^{-1}
\equiv \zeta \mod n\}$  can be partitioned as a finite union of precells. \end{proof}


\subsection{Example: the language $\Ldist$}
We will now use the results from the previous section to investigate the language $\Ldist$.
More specifically, we will show that the definitional expansion
\[\Ldist':= (\{\Dvierk\}_{k\in\Z}, \{R_{n,m}\}_{n,m}),\]
where
\[\Dvierk(x,y,z,t) \leftrightarrow \ord (x-y) < \ord (z-t) + k,\]
admits elimination of quantifiers for any $(\F_q,\Z)$-field.

\begin{definition}
We call a polynomial $f(x) \in K[x_1, \ldots, x_k]$ an  $\Lm_{\text{dist}}$-polynomial in variables $\{x_1, \ldots x_k\}$ if $f(x)$ has one of the following forms
\[ f(x)=a,\quad \text{or} \quad f(x) = \pi^k(x_i-a), \quad \text{or} \quad f(x) = \pi^k(x_i-x_j),\]
where $a \in K,\ k \in \Z$; $1\leqslant i,j \leqslant k$.  
\end{definition}
An $\Ldist$-cell will be a $p$-adic cell with the following specifications.
\begin{definition}
Let $\Delta_{\text{dist},k}$ be the set of all $\Ldist$-polynomials in $k$ variables, and put $\Deltadist := \cup_{k\geqslant 0}\, \Delta_{\text{dist},k}$.\\
An $\Ldist$-cell $\subseteq K^{k+1}$ is an $(\Ldist, \Deltadist)$-definable cell $\subseteq K^{k+1}$ for which the center $c(x)$ is either a constant from $K$ or one of the variables $x_1, \ldots x_k$.\end{definition}
It is an easy consequence of Lemma \ref{lemma:cd-defrel} that $\Ldist$- polynomials are order-definable functions. Therefore the following holds for $\Ldist$-cells.
\begin{proposition} \label{cd-prop:ldistcd}
 Let $A_1$, $A_2$ be $\Ldist$- cells with centers $c_1$, resp. $c_2$.
The intersection $A_1 \cap A_2$ can be written as a finite union of disjoint
$\Ldist$-cells $A$ with as center a restriction of either $c_1$ or $c_2$.
\end{proposition}
\begin{proof}
This follows from Theorem \ref{theorem:lmcell-doorsnede}.
\end{proof}
\begin{proposition}\label{prop:ldistqe}
The language $\Ldist'$ admits elimination of quantifiers for  any $(\F_q,\Z)$-field.
\end{proposition}
\begin{proof}
It is clear that any $\Ldist$-cell (and any $\Ldist$-precell) is quantifier free definable in $\Ldist'$. Moreover, for any $\Ldist$-polynomial $f(x)$, the relation $\ord f(x) \equiv l \mod n$ can be partitioned in a finite number of precells. (Indeed, this relation can be written as a finite disjunction of relations $x_i -a \in \lambda Q_{n,m}$ or $x_i-x_j \in \lambda Q_{n,m}$.) \\\\Since the requirements of Lemma \ref{cd-lemma:qe} are satisfied, it is now sufficient to show that any set that is quantifier free definable in $\Ldist'$ can be partitioned as a finite union of $\Ldist$-cells.

By Proposition \ref{cd-prop:ldistcd}, we only need to check 
that the sets (and complements of these sets)\[\{x \in K^k \ | \ \Dvier_r(g_1,g_2, g_3, g_4)\} \quad \text {and} \quad \{ x \in
K^k \ | \ R_{n,m}(g_1, g_2, g_3)\},
\]with $g_i(x) \in \{x_1, \ldots, x_k\} \cup K$ and $r\in \Z$ can be partitioned as a finite union of $\Lm_{\text{dist}}$-cells. \\The fact that $R_{n,m}(x,y,z)$ is equivalent to
\[\bigvee_{[\lambda] \, \in \, \Lambda_{n,m}}R_{n,m}(x,y,\lambda) \wedge R_{n,m}(0,z,\lambda),\]  implies
that $\{ x \in K^3 \ | \ R_{n,m}(x,y,z)\}$ can be written as a union of $\Lm_{\text{dist}}$-cells, by
Theorem \ref{theorem:lmcell-doorsnede}. The complement of this set
can also be written as a union of disjoint $\Lm_{\text{dist}}$-cells, since
$\{(x,y) \in K^2 \ | \ \neg R_{n,m}(x,y,\lambda)\}$ can be
written as a finite union of (disjoint) sets of the form $\{(x,y) \in K^2 \ | \ R_{n,m}(x,y,\mu)\}$.\\\\
To complete the proof it suffices to check  that the set
\[A:=\{(x,t) \in K^{k+1} \ | \ \ord(t-c_1) < \ord
\pi^n(t-c_2)\},\]with $c_1, c_2 \in \{ x_1, \ldots, x_k\} \cup K$,
can be partitioned as a finite union of cells. We may
suppose that $c_1 \neq c_2$. Partition $K^{k+1}$ in the
following way:
\begin{eqnarray}K^{k+1} &=& \{(x,t) \in K^{k+1}\ |\
\ord(t-c_1) > \ord (c_1-c_2) \}\nonumber\\
&  & \cup \ \{(x,t) \in K^{k+1}\ |\
\ord(t-c_1) < \ord (c_1-c_2) \}\label{partitionqp}\\
& & \cup \ \{(x,t) \in K^{k+1}\ |\ \ord(t-c_1) = \ord (c_1-c_2)
\}.\nonumber
\end{eqnarray}
Since $A = A \cap K^{k+1}$, we can write $A$ as a union of sets
on which one of the conditions in (\ref{partitionqp}) holds. For
example, on
\[B = A \cap\{(x,t) \in K^{k+1}\ |\ \ord(t-c_1) >
\ord (c_1-c_2) \},\] we have that $\ord(t-c_2) = \ord(c_1-c_2)$,
and therefore $B$ is equal to the set \[ B = \{(x,t) \in
K^{k+1} \ | \ord(c_1-c_2) < \ord(t-c_1) < \ord
\pi^n(c_1-c_2)\}.\] It is clear that $B$ can be partitoned as a finite number of $\Ldist$-cells. The other cases are similar.
\end{proof}

Note our strategy: For the given language $\Lm$, we first try to find a suitable set $\Delta$ of order-definable functions, such that we get a system of weak $(\Lm, \Delta)$-definable cells.
To obtain a definitional expansion that has QE, we add symbols to $\Lm$ such that for each $f \in \Delta$, the relations \[\ord f(x) < t +k \quad \text{and} \quad \ord f(x) \equiv l \mod n\] are quantifier free definable in the extended language $\Lm'$. To obtain QE for $\Lm'$, it is then sufficient to show that quantifier free $\Lm'$-definable sets can be partitioned as a finite number of $(\Lm,\Delta)$-cells. (Unfortunately, this last step may require quite a lot of work.) 

\section{Cell decomposition and definable (Skolem) functions } \label{section:deffun}
\subsection{Definable functions}
The example of $\Ldist$-definable sets illustrates how we can use cell decomposition to obtain quantifier elimination results for a structure $(K, \Lm)$. Cell decomposition results also provide a lot of information concerning the definable functions of a given structure. 
%
For instance,  
 all $\Ldist$-definable functions must have the following form (and thus these structures have only trivial definable functions):

\begin{lemma}\label{lemma:deffun-ldist}
Let $f: A \subseteq K^k \to K^l$ be an $\Lm_{\text{dist}}$-definable
function. There exists a finite partion of $A$ in $\Ldist$-cells such that
on each cell $C$ the function $f$ has the form
\[f_{|C} : C \to K^l: x \mapsto (f_1(x), f_2(x), \ldots,
f_n(x)),\] where $f_i(x)$ is either one of the variables $\{x_1, \ldots ,x_k\}$ or a constant from $K$.
\end{lemma}
\begin{proof}
 First we note that $f$ can be written as
\[f: A \subseteq K^ k \to  K^l: x \mapsto (f_1(x), \ldots, f_n(x)),\]
where the coordinate functions $f_i : K^m \to K$ are all $\Lm_{\text{dist}}$-definable functions. 
Therefore it is enough to prove the lemma for the case $l =1$.
A function $f: A \subseteq K^k \to K$ is $\Lm_{\text{dist}}$-definable if and only if
\[\mathrm{Graph} f = \{(x,t) \in A \times K \ | t = f(x) \} \] is an $\Lm_{\text{dist}}$-definable set.
This means there exists a finite partition of Graph $f$ in
$\Lm_{\text{dist}}$-cells $G$ of the form
\[  \{(x,t) \in D\times K \ | \ \ord a_1(x) \ \square_1\  \ord( t-c(x) )\
\square_2\ \ord a_2(x),\ t-c(x) \in \lambda Q_{n,m} \},\]with  $c(x)$ a constant from $K$ or
a variable from $\{x_1, \ldots, x_k\}$. But since $f$ is a function, for each $x\in D$
there must be a unique $t$ such that $(x,t) \in G$. This
uniqueness condition implies that $\lambda=0$, and thus $G$ must
have the form
\[G = \{ (x,t) \in D \times K \ | \ t = c(x) \}.\]
\end{proof}
\noindent
\subsection{Cell decomposition and Skolem functions}

 When studying definable functions, another natural question to ask is whether a language $\Lm$ has definable Skolem functions: for a given definable function $f: X\to Y$, does there exist a definable function $g: f(X) \to X$ such that $f\circ g = \text{Id}_{f(X)}$? 

In the $P$-minimal context, Mourgues showed that a $P$-minimal structure $(K,\Lm)$ has definable Skolem functions if and only if  the structure allows cell decomposition, using a notion of cells similar to what we called `strong cells', i.e. using cells that are defined using only definable functions.
When $K$ is $p$-adically closed, this result can be extended to extensions of $(K, \Ldist)$, if scalar multiplication by elements of the following set is definable:

\begin{lemma-def}
 The algebraic closure $\overline{\Q}^K$ of $\Q$ in $K$ is the field containing all elements of $K$ that are algebraic over $\Q$.\\ 
 The field $\overline{\Q}^K$ has the same residue field as $K$. Let $\ord_K: K \to \Gamma_K$ be the valuation on $K$. There exists $\pi \in \overline{\Q}^K$ such that $\ord_K(\pi)=1$.
\end{lemma-def}
\begin{proof}
 It is easy to see that $\overline{\Q}^K$ has the same residue field as $K$, since each $x \in \F_K$ is a simple root of $X^{q_K}-X$. The claim follows then by Hensel's Lemma.
 That $\overline{\Q}^K$ contains an element $\pi$ with $\ord_K(\pi)=1$, follows by Lemma 3.5 of \cite{pre-ro-84}.
 \end{proof}

\begin{theorem}
Let $K$ be a $p$-adically closed field. Suppose that $\Lm \supseteq \Ldist$ and that multiplication by constants from $\overline{\Q}^{K}$ is definable in $(K, \Lm)$. The structure $(K,\Lm)$ admits strong cell decomposition
if and only if $(K, \Lm)$ has definable Skolem functions.
\end{theorem}
In fact, we actually obtain strong cell decomposition using continuous functions, since every structure is a reduct of a $P$-minimal structure.
\\\\
\begin{proof}
First assume that $(K, \Lm)$ has definable Skolem functions. 
This implies that, if $(K,\Lm)$ has cell decomposition, say using $(\Lm, \Delta)$-cells, then it admits strong cell decomposition. Indeed, 
let $f: K^k \to K$ be a function in $\Delta$. Since $f$ is order-definable in $\Lm$, the following set is $\Lm$-definable:
\[ A:= \{ (x,t) \in K^{k+1} \ | \ \ord f(x) = t \}. \]
Consider the projection map \ 
\( \pi_x: A\to K^k: (x,t) \mapsto x.\)
Since $(\Lm,K)$ has definable Skolem functions,  there exists a definable function $g: \mathrm{Im}\, \pi_x \to A: x \mapsto (x,a(x))$ such that $(\pi_x \circ f)(x) = x$ for all $ x \in \mathrm{Im}\, \pi_x$. But then $a(x)$ is a definable function, such that for each $x \in K^k$, \(\ord f(x) = \ord a(x).\)

We can now show that $(K, \Lm)$ has strong cell decomposition, using essentially the same proof as that of Mourgues. For this reason, we will only give a brief sketch of the proof, and refer the reader to \cite{mou-09} for details.
Let $\Lm$ be an extension of $\Ldist$, and $S' \subseteq K^{n+1}$ an $\Lm$-definable set, defined by a formula $\phi(y,x)$. As in \cite{mou-09}, using a compactness argument, it can be shown that there exists a quanifierfree $\Ldist$-formula $\psi(z,x)$ such that
\begin{equation} K \models \forall y \exists z \forall x (\phi(y,x) \Leftrightarrow \psi(z,x)). \label{eq:mouskolem}\end{equation}
Write $\pi_n: K^{n+1} \to K^n$ for the projection onto the first $n$ coordinates.  As in Lemma 3.3 of \cite{mou-09},  it can follows from \eqref{eq:mouskolem} that, if $(K,\Lm)$ has definable skolem functions, there exists $m$, an $\Ldist$-definable subset $S$ of $K^{m+1}$ and an $\Lm$-definable function $f: \pi_n(S') \to K^m$ such that for any $y \in \pi_n(S')$,
\[ \{x \in K \mid (y,x) \in S'\} = \{ x \in K \mid (f(y),x) \in S \}.\]
Now use the same reasoning as in the proof of Theorem 3.5 of \cite{mou-09}, reducing to $\Ldist$-cell decomposition instead of semi-algebraic cell decomposition. The only thing that needs some care is to check that the cell decomposition obtained is a strong decomposition, but this is an immediate consequence of the observation at the start of this proof.

Next, assume that every function in $\Delta$ is $\Lm$-definable.
If suffices to check that 
given an ($\Lm, \Delta$)-cell $C$ and the projection map $\pi_x: C \subset K^{l+1} \to K^l$,
there exists a definable function $g: \pi_x(C) \to C$ such
that $\pi_x \circ g = \mathrm{Id}_{\pi_x(C)}$.
\\
If the cell $C$ has a center $c(x) \neq 0$, we first apply a
translation
\[C\to C': (x,t) \mapsto (x,t-c(x)),\]
to a cell $C'$ with center $c'(x)=0$. Since this translation is
bijective, it is invertible. Therefore the problem is reduced to the following. 
Let $C$ be a cell of the form
\[C = \{(x,t) \in D \times K \ | \ \ord b(x)\,\square_1 \,\ord
t \,\square_2 \,\ord a(x), \ t \in \lambda Q_{n,m} \},
\] where $a(x), b(x)$ are $\Lm$-definable functions and
$D$ is an $(\Lm,\Delta)$-precell. 
We must show that there exists a definable function $g: \pi_x(C) \to C$ such
that $\pi_x \circ g = \mathrm{Id}_{\pi_x(C)}$. \\\\
Given $x \in \pi_x(C)\subseteq D$, we have to find $t(x)$ such that
$(x,t(x))$ satisfies the conditions
\begin{eqnarray}
 \ord b(x)\ \square_1 \ \ord
t(x) \ \square_2 \ \ord a(x) \label{c1}\\
t(x) \in \lambda Q_{n,m} \label{c2}
\end{eqnarray}
If $\lambda = 0$, put $g(x) = (x,0)$. From now on we assume that $\lambda \neq 0$.\\
If $\square_1 = \square_2 =$ `no condition', we can simply put $g(x)
=(x, \lambda).$\\
If $\square_2 = $ $<$, we can define $g$ as follows. First partition
$\pi_x(C)$ in parts $D_{\mu}$, such that
\[ D_{\mu} = \{x\in \pi_x(C) \ | \ a(x) \in \mu Q_{n,m} \}.\]
(Note: if $\mu = 0$, we can reduce to the cases were $\square_2$ = `no condition'.)
Our strategy is based on the fact that for every $x \in D$, there
exists $k \in \Z$ such that $k$ satisfies
\[\ord b_1(x) \ \square_1 \ord \lambda + kn < \ord a(x). \]
Restricting to a set $D_{\mu}$, we construct an element $t(x)$
with order as close as possible to $\ord a(x)$. This ensures that
$t(x)$ satisfies (\ref{c1}).
The definiton of $g$ on $D_{\mu}$ will depend on
the respective orders of $\lambda$ and $\mu$.
\begin{itemize}
\item If $\ord \lambda < \ord \mu$, we can define $g_{|D_{\mu}}$ as
\(g_{|D_{\mu}}: D_{\mu} \to C: x \mapsto \left(x,\frac{\lambda}{\mu} a(x)\right).\)
This means that we put $t(x) = \frac{\lambda}{\mu} a(x)$. Clearly
 $t(x) \in \lambda Q_{n,m}$. Also, since $-n < \ord
(\frac{\lambda}{\mu}) < 0$, we have that $0<\ord
\frac{a(x)}{t(x)}<n$, and thus condition (\ref{c1}) must be
satisfied.
\item If $\ord \lambda \geqslant \ord \mu$, put
\(g_{D_{\mu}}: D_{\mu} \to C: x \mapsto \left(x,\frac{\lambda}{\pi^n\mu} a(x)\right).\)
\end{itemize}
If  $\square_1 =$ $<$ and $\square_2 =$ `no condition', we choose
$t(x)$ with order as close as possible to $\ord b(x)$. More
specifically, if $\ord \lambda \leqslant \mu$, define $g$ as
\(g_{D_{\mu}}: D_{\mu} \to C: x \mapsto \left(x,\frac{\lambda \pi^n}{\mu}
b(x)\right)\), and if $\ord \lambda > \ord \mu,$ put
\(g_{D_{\mu}}: D_{\mu}\to C: x \mapsto \left(x,\frac{\lambda}{\mu} b(x)\right).\)
\end{proof}

If we omit the condition that $K$ has to be $p$-adically closed, we obtain a weaker version of this theorem: we can no longer be assured that definable skolem functions imply the existence of cell decomposition, but if such a decomposition exists, it will be a strong decomposition. 

As before, Skolem functions will only be definable if we can define multiplication by enough scalars. 
Write $\PP_K$ for the prime field of $K$. Choose a generator $\pi_K$ of $R_K$. If the residue field $\F_K = \F_p[\overline{a_1}, \ldots, \overline{a_d}]$,  choose  elements $a_i \in R_K$ such that $\ord(a_i - \overline{a_i}) >0$. Then put $\K_K:=\PP_K[\pi_k, a_1, \ldots, a_{d}]$.

\begin{corol}
 Suppose that $\Lm \supseteq \Ldist$ and that multiplication by constants from $\K_K$ is definable in $(K, \Lm)$. \\Let $\Delta$ be a collection of order-definable functions such that the structure $(K, \Lm)$ has cell decomposition using $(\Lm, \Delta)$-cells. The following statements are then equivalent:
\begin{enumerate}
\item There exists a collection of $\Lm$-definable functions $\Delta'$ such that $(K, \Lm)$ has cell decomposition using $(\Lm, \Delta')$-cells, 
\item The structure $(K, \Lm)$ has definable Skolem functions.
\end{enumerate}
\end{corol}

\noindent The condition that multiplication by constants from $\K_K$  be definable in $(K, \Lm)$, is needed: for example the structure $(K; +,-, D^{(4)}, \{R_{n,m}\}_{n,m})$ does not have definable Skolem functions for most $(\F_q, \Z)$-fields because we cannot define multiplication by enough scalars (in this structure, scalar multiplication is only definable for elements of $\PP_K$).

%

\section{Subaffine structures}\label{section:subaffine}
In this section we study some expansions of the language $\Ldist$ (or rather $\Ldist'$ as we would like to achieve quantifier elimination whenever possible). We call these expansions sub\textbf{affine} because we will only be considering structures ($K, \Lm$)
 that are affine in the sense that there does not exist any open subset of $K^2$ on which multiplication is $\Lm$-definable. They are \textbf{sub}affine because addition should not be definable on all of $K^2$.
 \\\\
 A first, rather trivial example of such an expansion is the language we obtain by adding symbols $\overline{c}$ for the scalar multiplication $\overline{c}: x\mapsto cx$.
 \[\Lm_K := \{\overline{c}\}_{c\in K} \cup \Ldist'.\]
 Take fields $F, K$ and  $q,q' \in \N$  such that $F$ is an $(\F_{q},\Z)$-field, $K$ is an $(\F_{q'},\Z)$-field and   $F \supset K$. Define the set $\Delta_{K,F}$ to be $\bigcup_{k\in \N} \Delta_{K,F}^{k}$, where $\Delta_{K,F}^{k}$ is the following set of polynomials
 \[\Delta_{K,F}^{k} := \{ax+by \ | \ a,b \in K; x,y \in \{x_1, \ldots, x_k\}\cup F\}.\]
 It is easy to see that the structure ($F, \Lm_K$)   has cell decomposition and quantifier elimination using $(\Lm_K, \Delta_{K,F})$-cells. The proof is almost literally the same as the corresponding proof for $\Ldist$. Moreover,  the scalar multiplication functions we added are in fact the only non-trivial functions, or to be more precise:
 
\begin{lemma}\label{lemma:deffun-lscal}
Let $f: A \subseteq F^k \to F^l$ be an $\Lm_{K}$-definable
function. There exists a finite partition of $A$ in $(\Lm_K, \Delta_{K,L})$-cells such that
on each cell $C$ the function $f$ has the form
\[f_{|c} : C \to K^l: x \mapsto (f_1(x), f_2(x), \ldots,
f_l(x)),\] where $f_i(x)$ is either a constant from $F$ or $f_i(x) = ax_j$, with $a \in K$ and $x_j$ one of the variables $\{x_1, \ldots ,x_k\}$.
\end{lemma}
\noindent Structures  $(F, \Lm_K)$ do not have definable Skolem functions, as can be seen from the following counterexample:
\begin{lemma}
Let $\Pi$ be the projection map 
\[ \Pi: A:= \{(x,y,z) \in F^3 \ | \ \ord z = \ord(y-x) \} \to F^2 : (x,y,z) \mapsto (x,y).\]
There exists no $\Lm_{K}$-definable function $f$ such that $\Pi \circ f = \mathrm{Id}_{\mathrm{Im}_{\Pi}}$.
\end{lemma}
\begin{proof}
Suppose that such a function $f$ exists.  Up to a finite partition of $\pi(A)$ in cells $\bigcup C_i\cup \bigcup D_j$, this function will be of one of the forms 
\[f_{|C_i}: C\to F^3: (x,y)\mapsto (x,y, a_i) \ \text{or} \ f_{|D_j}: D_j\to F^3: (x,y)\mapsto (x,y, b_jx_j),\]
with $a_i \in F$, $b_j\in K$ and $x_j$ is one of the variables $x$ and $y$. \\
On cells $C_i$, we use a function of the form $(x,y) \mapsto (x,y,a_i)$, which implies that $\ord (x-y) = \ord a_i$ for all $(x,y) \in C_i$. As our partition is finite, we can only have a finite number of cells of this type. Put $M := \max_i \ord a_i$, then all tuples $(x,y)$ for which $\ord (x-y)>M$ will be contained in $\bigcup D_j$.  For each $k>0$, this set contains elements $(x,y)$ that satisfy
\[\ord x =\ord y < M \wedge \ord x-y = \ord x +k,\]
which means that we would need a partition in an infinite number of parts $D_j$ to define $f$.
\end{proof}

This counterexample suggests that it might be impossible to have definable Skolem functions in a language where addition is not definable. This is our main motivation for studying subaffine structures: we will consider a language that has a restricted form of addition, and see whether this languages allows us to define Skolem functions.   We will consider the following the functions $\PlusR$ and $\MinR$, defined by 
 \[\PlusR: K^2 \to K : (x,y) \mapsto 
\left\{\begin{array}{ll} 
x + y & x,y \in R_K\\
0& \text{otherwise,}
\end{array}\right.\]\\
and analogously for $\MinR$, with $+$ replaced by $-$.


For these functions, we will study the language
\[\Lm_{\boxplus,K} := \{\boxplus, \boxminus\} \cup \Lm_K.\]

Let $F \supset K$ be  $(\F_q, \Z)$-fields, resp. $(\F_{q'}, \Z)$-fields. 
We will verify that a structure $(F, \Lm_{\PlusR,K})$ has cell decomposition and quantifier elimination for the language
\[\Lm_{\PlusR,K} := \{\PlusR, \MinR\}\cup \Lm_K.\]
\begin{definition} Let $(F, \Lm_{\PlusR,K})$ be a structure.\\
Write $\mathrm{Poly}_{\PlusR,K}$ for the set of functions that can be defined as a composition of the functions $\PlusR,\MinR$ and $\overline{c}$ for $c \in K$, combined with variables $x_1, x_2, \ldots$  and constants from $F$. 
\end{definition}
\noindent It is important to stress that  these expressions do not entirely behave like polynomials. More precisely,  we have to be aware that distributivity does not always hold. For example: suppose $0<k_1<k_2<k_3$, then
\[\pi^{-k_2}(\pi^{k_1} \PlusR \pi^{k_3}) \neq \pi^{k_1-k_2} \PlusR \pi^{k_3-k_2} = 0.\] 
First we need to define a notion of cells for this context. 
\begin{definition}
Let $\Delta_{\PlusR, K}$ be the set
\[\bigcup_{r\in \N}\{ a(x_1,\ldots, x_r)-b(x_1, \ldots, x_r) \ | \ a(x), b(x) \in \mathrm{Poly}_{\PlusR,K} \}.\] 
A subset of $F^k$ is called a $(\PlusR,K)$-cell if it is a $(\Lm_{\PlusR,K}, \Delta_{\PlusR,K})$-cell and the center is a function from $\mathrm{Poly}_{\PlusR,K}$.
\end{definition}

\noindent In the next lemmas, we will show that, up to a finite partition in cells,  $\mathrm{Poly}_{\PlusR,K}$-functions can always be written in a fairly simple way. 
Note that for every $\gamma_0 \in \Gamma_K$, the following function is (quantifier free) definable:
\[\boxplus_{\gamma_0}: (x,y) \mapsto \left\{\begin{array}{ll}
x+y & \ord x,\ord y \geqslant \gamma_0\\
0 & \text{otherwise} 
\end{array}\right.\]
Moreover, we have the following calculation rule. For every $a\in K; b,c \in F, \gamma \in \Gamma_K$:
\[ a(b \boxplus_{\gamma} c) = ab \boxplus_{\gamma + \ord a} ac.\]

   \begin{lemma} \label{lemma:Sincells}
   Take $\gamma \in \Gamma_K, a \in K;\  d(x), h(x) \in\mathrm{Poly}_{\PlusR, K}$. Let $\square$ denote `$<$', `$\leqslant$', `$>$' or `$\geqslant$'.
   The set \[S:=\{(x,t) \in F^{k+1} \ | \ \ord(at \boxplus_{\gamma} d(x)) \ \square\ \ord h(x)\}\] can be partitioned as a finite union of $(\PlusR, K)$-cells.  \end{lemma}
   \begin{proof}
   First note that we may suppose that $a =1$ (if $at \neq 0$), since 
   \[\ord(at \boxplus_{\gamma} d(x)) \ \square\ \ord h(x) \Leftrightarrow \ord\left(t \boxplus_{\gamma - \ord a} \frac{d(x)}{a}\right) \square\ \ord \frac{h(x)}{a}.\]
   Put $a=1$. The set $S$ can then be partitioned as the union of the following three sets:
   \begin{eqnarray*}
   S &=& \{(x,t) \in F^{k+1} \ | \ \ord t < \gamma \wedge \ord 0 \,\square\, \ord h(x) \}\\
   && \cup \ \{(x,t) \in F^{k+1} \ | \ \ord t \geqslant \gamma \wedge \ord d(x) < \ord \gamma \wedge \ord 0 \,\square\, \ord h(x)\}\\
   && \cup \left(\{(x,t) \in F^{k+1} \ | \ \ord t \geqslant \gamma \wedge \ord d(x) \geqslant \gamma\}\right. \\ && \hspace{1cm}\left.\cap \ \{(x,t) \in F^{k+1} \ | \ \ord (t + d(x))\ \square\ \ord h(x) \}\right)
   \end{eqnarray*}
   The first two sets are cells. The third set is the intersection of two $(\PlusR, K)$-cells and thus again a finite union of cells by Theorem \ref{theorem:lmcell-doorsnede}.     \end{proof}
\begin{lemma}\label{lemma:Rplus-functionpartition}
For $x = (x_1, \ldots, x_n)$, and $t$ one variable, let the functions $f_1(x,t), \,$ $\ldots, \, f_r(x,t)$ be  in  $\mathrm{Poly}_{\PlusR,K}$. 
$F^{k+1}$ can be partitioned in a finite number of cells $A$, such that on each cell $A$ there are  $\gamma_1, \ldots, \gamma_r \in \Gamma_K$, such that either
 \[f_i(x,t) = h_i(x) \quad \text{or}\quad f_i(x,t) = a_it 
 \mathor  f_i(x,t) = a_it \boxplus_{\gamma_i} d_i(x),\]
 with $a_i \in K$, and $h_i(x), d_i(x)$ are in  $\mathrm{Poly}_{\PlusR,K}$. 
 \end{lemma}.
 
 \begin{proof} We will work by induction on the number of compositions. Suppose the lemma holds for functions $f$ and $g$. It suffices to check that the lemma also holds for $\overline{c} \circ f$ and $f\PlusR g$. Take a suitable decomposition into cells $A$. Choose $c \in K$. 
For all $(x,t) \in A$, the function $\overline{c} \circ f$ will have one of the following forms: either
 \begin{equation}(\overline{c} \circ f)(x,t) = \overline{c} \circ h_f(x) \quad \text{or}\quad (\overline{c} \circ f)(x,t) = ca_ft ,
 \label{eq:boxplus1}\end{equation}or 
 \begin{equation}  (\overline{c} \circ f)(x,t)=c(a_f t\boxplus_{\gamma_f} d_f(x)).\label{eq:boxplus2}\end{equation}
 If we have functions as in \eqref{eq:boxplus1}, we are done. We can rewrite \eqref{eq:boxplus2} as
 \[ c(a_ft \boxplus_{\gamma_f} d_f(x)) = ca_ft \boxplus_{\gamma_f +\ord c} (\overline{c} \circ d_f(x)).\]
We can apply a similar reasoning to the function $(f \PlusR g)$. In most cases, it is obvious that the function has one of the required forms. The only nontrivial cases are when $(f \PlusR g)$ has one of the following forms for $(x,t) \in A$: 
   \[(f \PlusR g)(x,t) = \left\{\begin{array}{ll}
  a_ft \PlusR (a_gt \boxplus_{\gamma_g} d_g(x)) & \text{(case 1)}\\
  h_f(x) \PlusR (a_gt \boxplus_{\gamma_g} d_g(x)) & \text{(case 2)}\\
  (a_ft \boxplus_{\gamma_f} d_f(x))\PlusR(a_gt \boxplus_{\gamma_g} d_g(x)) & \text{(case 3)}   \end{array}\right.\]
  Remember that the set $\{(x,t) \in K^{k+1} \ | \ \ord (a t \boxplus_{\gamma} d(x)) \geqslant 0\}$ can be written as a finite union of cells, by Lemma \ref{lemma:Sincells}.\\  We will check that the lemma holds in case 3 (Case 1 and 2 are similar). Partition $A$ further in cells such that either $\ord (a_ft \boxplus_{\gamma_f} d_f(x)) < 0$, or $\ord (a_ft \boxplus_{\gamma_f} d_f(x)) \geqslant 0$ for all $(x,t) \in A$ (and similarly for $g$). We only need to consider cells where $\ord (a_ft \boxplus_{\gamma_f} d_f(x)) \geqslant 0$ and $\ord (a_gt \boxplus_{\gamma_g} d_f(x)) \geqslant 0$ as our claim is trivially true on other cells. Partition these cells further depending on the order of $a_ft, a_gt, d_f(x), d_g(x)$. The only case that is not immediately obvious is when \[\ord a_ft \geqslant \gamma_f, \ord d_f(x) \geqslant \gamma_f, \ord a_gt \geqslant \gamma_g \text{\  and\  } \ord d_g(x) \geqslant \gamma_g.\] Let $C$ be such a cell. Without loss of generality, we may suppose that $\gamma_f \leqslant \gamma_g$. For $(x,t) \in C$ we find that
  \begin{eqnarray*}
  (f \PlusR g)_{|C}
  &=&(a_ft \boxplus_{\gamma_f} d_f(x))\PlusR(a_gt \boxplus_{\gamma_g} d_g(x)) \\ 
  &=& (a_f+ a_g)t + (d_f(x) + d_g(x))\\
  &=& (a_f+ a_g)t \boxplus_{\gamma_f} (d_f(x) \boxplus_{\gamma_f} d_g(x)).
    \end{eqnarray*}
   \end{proof}


\begin{proposition} \label{prop:qfd=cell_for_PlusR}
Any $\Lm_{\PlusR}$-definable set can be partitioned as a finite union of $(\PlusR,K)$-cells.
\end{proposition}   \begin{proof}
   First we show that quantifier-free definable sets can be partitioned as a finite union of cells. Because of Theorem \ref{theorem:lmcell-doorsnede}, it is sufficient to check that sets of type $S_1$ or $S_2$ can be partitioned as a finite union of cells:
   \begin{eqnarray*}S_1&:=&\{(x,t) \in F^{k+1} \ | \ f_1(x,t)-f_2(x,t) \in \lambda Q_{n,m}\}, \\S_2&:=& \{(x,t) \in F^{k+1} \ | \ \ord f_1(x,t) - f_2(x,t) < \ord f_3(x,t)-f_4(x,t)\},\end{eqnarray*}
   where the $f_i(x,t)$ are functions from $\mathrm{Poly}_{\PlusR, K}$.
   Intersect the sets $S_1$ and $S_2$ with sets $\{(x,t) \in F^{k+1} \mid \ord f_i(x,t) \ \square \ \ord f_j(x,t) \}$ or $\{(x,t) \in F^{k+1} \mid \ord f_i(x,t) \ \square \ \ord h_j(x) \}$,  where  $\square$ may denote $<, =$ or $>$. When we apply Theorem \ref{theorem:lmcell-doorsnede} and Proposition \ref{lemma:Rplus-functionpartition} to these intersections, it is easy to see that
%
%
   it suffices to check that the sets $\widetilde{S_1}$ and $\widetilde{S_2}$
   \begin{eqnarray*}\widetilde{S_1}&:=& \{(x,t) \in F^{k+1} \ | \ a_1t \boxplus_{\gamma_1} d_1(x) \in \lambda Q_{n,m}\}, \\ \widetilde{S_2}&:=& \{(x,t) \in F^{k+1} \ | \  \ord(a_2t \boxplus_{\gamma_2} d_2(x)) < \ord(    a_3t \boxplus_{\gamma_3} d_3(x))\},\end{eqnarray*}
   can be partitioned as a finite union of cells for all $\gamma_i \in \Gamma_K$ and $d_i(x) \in \mathrm{Poly}_{\PlusR,K}$.
   For $\widetilde{S_1}$ this follows from the observation that the expression ${at\boxplus_{\gamma} d(x) \in \lambda Q_{n,m}}$ is equivalent with
   \begin{align*} &\ [\ \ord at < \gamma \wedge \lambda = 0] \vee [\ord at \geqslant \gamma \wedge \ord d(x) < \gamma \wedge \lambda = 0]\\
  \vee&\left[\ \ord at \geqslant \gamma \wedge \ord d(x) \geqslant \gamma \wedge \left(t-\frac{-d(x)}{a}\right) \in \frac{\lambda}{a} Q_{n,m}\right].
    \end{align*}
    For the set $\widetilde{S_2}$, note that we can restrict our attention to $\widetilde{S_2}^{(1)}: = \widetilde{S_2} \cap  A_{\geqslant}$, with
    \[A_{\geqslant} := \{(x,t) \in F^{k+1} \ |   \ord a_it \geqslant \gamma_i \wedge \ord d_i(x) \geqslant \gamma_i, \text{\ for\ } i \in \{2,3\} \},\]
    since it follows easily from Theorem \ref{theorem:lmcell-doorsnede} that $\widetilde{S_2} \backslash \widetilde{S_2}^{(1)}$ can be partitioned as a finite union of cells. Write $d_i'(x) = \frac{-d_i(x)}{a_i}$. Now $\widetilde{S_2}^{(1)}$ is equal to the set
    \[\widetilde{S_2}^{(1)} = \{(x,t) \in A_{\geqslant} \ | \ \ord a_2(t - d_2'(x)) < \ord a_3(t -d_3'(x))\}. \]
    For $\square =$ `$<$', `$=$', or `$>$', put \[B_{\square}:= \{(x,t) \in K^{k+1} \ | \ \ord(t-d_2'(x)) \  \square \ \ord(d_2'(x) - d_3'(x) \}.\]     
    Then $\widetilde{S_2}^{(1)} = \left(\widetilde{S_2}^{(1)} \cap B_{<} \right) \cup   \left(\widetilde{S_2}^{(1)} \cap B_{=} \right) \cap  \left(\widetilde{S_2}^{(1)} \cap B_{>} \right)$.\\
    Now if $\ord (t-d_2'(x)) < \ord(d_2'(x) - d_3'(x))$, then $\ord (t - d_3'(x)) = \ord (t - d_2'(x))$, so 
    \[\widetilde{S_2}^{(1)} \cap B_{<} = A_{\geqslant} \cap B_{<} \cap \{(x,t) \in F^{k+1} \ | \ord a_2 < \ord a_3\}.\] 
    By Theorem \ref{theorem:lmcell-doorsnede}, this can be written as a finite union of cells. The situation is similar when we intersect with $B_{=}$ or $B_{>}$. 
    \\\\
    The fact that quantifier-free definable sets can be partitioned as a finite union of cells, also implies that for all $a_i \in \Delta_{\PlusR,K}$, the relation $\ord a_i \equiv l \mod n$ can be defined using precell relations. Because of this, structures $(F, \Lm_{\PlusR, K})$ have quantifier elimination by Lemma \ref{cd-lemma:qe}.
    
    \end{proof}
    \noindent The following classification of the definable functions is an immediate consequence of this proposition.
    \begin{corol}\label{corol:PlusR-functions}
    Let  $f: A \subseteq F^{\ell} \to F^{r}$ be an $\Lm_{\PlusR,K}$-definable function. There exists a finite partition of $A$ in cells $C$, such that on each cell $C$, 
    \[ f_{|C}(x) = (f_1(x), \ldots f_r(x)),\]
    where $f_i(x) \in \mathrm{Poly}_{\PlusR,K}$, for $i = 1, \ldots, r$.   \end{corol}
    
    \begin{proposition}
    The addition function $+: F^2\to F: (x,y) \mapsto x+y$ is not definable in $\Lm_{\PlusR,K}$.
    \end{proposition}
    \begin{proof}
    Suppose addition is definable, say by some function $f$. Applying Corollary \ref{corol:PlusR-functions} and Lemma \ref{lemma:Rplus-functionpartition}, We can partition $F^2$ in cells $C_i$ and $D_i$ such that 
  \[f_{|C_i}(x,y) = a_ix \boxplus _{\gamma_{i,1}} (b_iy \boxplus_{\gamma_{i,2}} c_i),\]
    and \[f_{|D_i}(x,y) = a_ix \mathor f_{|D_i}(x,y) = b_iy \mathor f_{|D_i}(x,y) = c_i,\]
    for some $a_i,b_i,\in K, c_i \in F$ and $\gamma_{i,1}, \gamma_{i,2} \in \Gamma_K$. The precise value of these constants depends on $C_i$. Put $\gamma := \min_i\{\gamma_{i,1}\}$.
    Clearly all elements $(x,y)$ for which $\ord a_ix < \gamma$ must be contained in the cells $D_i$ 
    since for such elements, 
    \[f_{|C_i}(x,y) = 0 \neq x+y.\]
    It is clear that the functions $f_{|D_i}$ cannot be used to define addition on a large enough set,
    which proves that  the addition function `$+$' is not definable.
    \end{proof}
    
 \noindent The fact that addition is not definable is caused by the fact that we have restricted multiplication to multiplication by a constant. More precisely, the reason is the following (for simplicity, suppose that $\Gamma_K$ is $\Z$). In our language, it is impossible to take a `limit' for $\ord x$ going to $-\infty$. In a language with normal multiplication, we do not have this restriction, and as a consequence `$+$' can easily be defined in such a language. Take for example the language $\Lm = (\PlusR, \cdot)$. For any $x, y \in K$ with $\ord x \leqslant \ord y$ and $x \neq 0$, we can define addition using the following equality.
    \[ x+y =x\left(1\PlusR \frac{y}{x}\right).\]
    \[\]
    It is also not hard to see that definable Skolem functions do not always exist for $\Lm_{\PlusR,K }$. Indeed, this follows from the following counterexample.
    \begin{lemma}
    Let $\Pi$ be the projection map 
    \[ \Pi: \{(x,y,z) \in F^3 \ | \ \ord z = \ord(y-x) \} \to K^2 : (x,y,z) \mapsto (x,y)\}.\]
    There exists no $\Lm_{\PlusR,K}$-definable function $f$ such that $\Pi \circ f = \mathrm{Id}_{\mathrm{Im}_{\Pi}}$.
    \end{lemma}
    \begin{proof}
    Put $A = \{(x,y,z) \in F^3 \ | \ \ord z = \ord(y-x) \}$.
    Suppose $f: \Pi(A) \to A$ is a definable function for which $\Pi \circ f = \mathrm{Id}_{\mathrm{Im}_{\Pi}}$. By Corollary \ref{corol:PlusR-functions}, there exists a partition of $\Pi(A)$ in cells $C_i$ and $D_i$ such that
    \[f_{|C_i}(x,y) = (x,y, a_ix \boxplus_{\gamma_i} h_i(y)),\]
    and 
     \[f_{|D_i}(x,y) = (x,y, a_ix ), \mathor f_{|D_i}(x,y) = (x,y, h_i(y)) ,\]
    with $a_i \in K, \gamma_i \in \Gamma_K$ and $h_i \in \mathrm{Poly}_{\PlusR,K}$.
    Note that $f_{|C_i}(x,y)= (x,y,0)$ for elements $(x,y)$ for which $\ord a_ix < \gamma_i$. So the sets $D_i$ must contain all $(x,y)$ for which $\ord x$ is too small. However, it is easy to see that the functions $f_{|D_i}$ do not satisfy our requirements.
    \end{proof}
    
    This is a confirmation of our conjecture that a structure where addition is not definable does not have definable Skolem functions. Take for example the structure $(F, \Lm_{\PlusR,F})$. If we fix any constant $\gamma \in \Gamma_F$, we can define addition for the set $\{(x,y) \in F^2 \ | \min\{\ord x, \ord y\} \geqslant \gamma\}$. Taking smaller and smaller values for $\gamma$, we can thus define addition on very large open subsets of $F^2$, but still not large enough to enable us to define Skolem functions.
    
    \subsection*{Acknowledgements}
The results presented in this paper were obtained as part of my PhD thesis. I would like to thank my supervisor, Raf Cluckers, for many stimulating conversations about this topic, and other members of the jury (in particular, Jan Denef, Angus Macintyre and Leonard Lipshitz) for useful comments. Many thanks also to the Math Department of K.U.Leuven, for providing financial support and a very stimulating working environment.
    
  \bibliographystyle{plain}

\bibliography{/Users/iblueberry/Documents/Bibliografie} 

\begin{thebibliography}{10}

\bibitem{bas-kuh-92}
{\c{S}}erban~A. Basarab and Franz-Viktor Kuhlmann.
\newblock An isomorphism theorem for {H}enselian algebraic extensions of valued
  fields.
\newblock {\em Manuscripta Math.}, 77(2-3):113--126, 1992.

\bibitem{clu-lee-2011}
R.~Cluckers and E.~Leenknegt.
\newblock A version of $p$-adic minimality.
\newblock {\em Journal of Symbolic Logic}, 77(2):621--630, June 2012.

\bibitem{denef-84}
J.~Denef.
\newblock The rationality of the {P}oincar\'e series associated to the $p$-adic
  points on a variety.
\newblock {\em Invent. Math.}, 77:1--23, 1984.

\bibitem{denef-86}
Jan Denef.
\newblock {$p$}-adic semi-algebraic sets and cell decomposition.
\newblock {\em J. Reine Angew. Math.}, 369:154--166, 1986.

\bibitem{has-mac-97}
Deirdre Haskell and Dugald Macpherson.
\newblock A version of o-minimality for the {$p$}-adics.
\newblock {\em J. Symbolic Logic}, 62(4):1075--1092, 1997.

\bibitem{kuh-94}
Franz-Viktor Kuhlmann.
\newblock Quantifier elimination for {H}enselian fields relative to additive
  and multiplicative congruences.
\newblock {\em Israel J. Math.}, 85(1-3):277--306, 1994.

\bibitem{lee-2011}
E.~Leenknegt.
\newblock Cell decomposition for semi-affine structures on $p$-adic fields.
\newblock {\em Submitted}.

\bibitem{PhD}
E.~Leenknegt.
\newblock {\em Cell decomposition for $p$-adic fields: definable sets and
  minimality}.
\newblock PhD thesis, K.U.Leuven, April 2011.

\bibitem{mac-76}
A.~Macintyre.
\newblock On definable subsets of $p$-adic fields.
\newblock {\em J. Symb. Logic}, 41:605--610, 1976.

\bibitem{mpp-92}
D.~Marker, Y.~Peterzil, and A.~Pillay.
\newblock Additive reducts of real closed fields.
\newblock {\em J. Symbolic Logic}, 57(1):109--117, 1992.

\bibitem{mou-09}
Marie-H{{\'e}}l{{\`e}}ne Mourgues.
\newblock Cell decomposition for {$P$}-minimal fields.
\newblock {\em MLQ Math. Log. Q.}, 55(5):487--492, 2009.

\bibitem{pet-93}
Y.~Peterzil.
\newblock Reducts of some structures over the reals.
\newblock {\em J. Symbolic Logic}, 58(3):955--966, 1993.

\bibitem{pet-92}
Ya'acov Peterzil.
\newblock A structure theorem for semibounded sets in the reals.
\newblock {\em J. Symbolic Logic}, 57(3):779--794, 1992.

\bibitem{pre-ro-84}
A.~Prestel and P.~Roquette.
\newblock {\em Formally $p$-adic fields}.
\newblock Lecture Notes in Mathematics. Springer-Verlag, Berlin, 1984.

\bibitem{vdd-98}
Lou van~den Dries.
\newblock {\em Tame topology and o-minimal structures}, volume 248 of {\em
  London Mathematical Society Lecture Note Series}.
\newblock Cambridge University Press, Cambridge, 1998.

\end{thebibliography}
\end{document}